\documentclass[12pt,a4paper]{amsart}
 \usepackage{amssymb}

 \usepackage{multirow}  
 \usepackage{hhline}
 \usepackage{ifthen}
 \usepackage{amscd}

 \theoremstyle{plain}
  \newtheorem{thm}{Theorem}[section]
  \newtheorem{lem}[thm]{Lemma}
  
  \newtheorem{cor}[thm]{Corollary}

\theoremstyle{definition}
  
  \newtheorem{conj}[thm]{Conjecture}
  \newtheorem{ex}[thm]{Example}

 \newtheorem{asmp}[thm]{Assumption}

\theoremstyle{remark}
  \newtheorem{rem}[thm]{Remark}

  \newtheorem{ack}{Acknowledgment}

\numberwithin{equation}{section}

\DeclareMathOperator{\diam}{diam}               % diam
\newcommand{\field}[1]{\mathbb{#1}}

\newcommand{\R}{\field{R}}                      % R
\renewcommand{\setminus}{-}

\begin{document}

% title.tex
\title{Collapsing and essential coverings }

\author{Takao Yamaguchi}

\address{Institute of Mathematics, University of Tsukuba, Tsukuba
 305-8571, JAPAN}

\email{takao@math.tsukuba.ac.jp}

%\date{April 3, 2003}



\begin{abstract}
  In the present paper, we consider the family of all compact Alexandrov spaces
  with  curvature bound below having a definite upper diameter bound
  of a fixed dimension.
  We introduce the notion of essential coverings by contractible 
  metric balls, and provide a uniform bound on the numbers of contractible 
  metric balls forming essential coverings of the spaces in the
  family. In particular, this gives another view for Gromov's 
  Betti number theorem.
\end{abstract}

\maketitle

\section{Introduction} \label{sec:intro}

It is an important problem to find the relation between 
curvature and topology of Riemannian manifolds.
In the study of the finiteness of compact Riemannian manifolds
with uniformly bounded curvature, originated by Cheeger \cite{Ch:finite} and 
Weinstein \cite{Ws:homotopy}, it was a crucial idea to cover 
a manifold in a certain class  by convex metric balls whose
number is uniformly bounded depending only on the class.
The minimal number of 
such metric balls covering a manifold represents the 
complexity of the manifold.
For a certain class of  compact Riemannian manifolds with a lower
curvature bound, Grove, Petersen and Wu \cite{GPW} used 
contractible balls in place of convex balls to get 
a topological finiteness of the manifolds in the class.
Those are results in the non-collapsing cases.

In the present paper, 
we consider the collapsing case for  compact 
Alexandrov spaces $M$
with curvature  bounded below.
The Perelman stability theorem \cite{Pr:alex2} shows 
that a small metric ball around  a given point 
of $M$ is homeomorphic to the tangent cone, and hence contractible.

If $M$ is  collapsed, the sizes of contractible metric balls must be
very small,
and therefore the minimal number of contractible metric balls 
covering $M$ becomes large. In other words, 
it is not efficient to cover the whole $M$ 
by  contractible metric balls. 
In the present paper,  to overcome this difficulty, 
we introduce the notion of 
an {\it essential covering} of $M$ in place of a usual covering.

To illustrate the notion of essential covering, let us take 
a flat torus $T^2_{\epsilon}= S^1(1)\times S^1(\epsilon)$ for a 
small $\epsilon>0$, 
where $S^1(\epsilon)$ is the circle of length $\epsilon$.
The torus $T^2_{\epsilon}$ can be covered by two thin metric balls
$B_{\alpha}$, $\alpha\in \{ 1, 2\}$.
Each ball $B_{\alpha}$ is isotopic to a much smaller
concentric metric ball $\hat B_{\alpha}$ of radius, say $2\epsilon$.
If one tries to cover $B_{\alpha}$ by contractible metric balls,
we need too many, about $[1/\epsilon]$-pieces of such balls.
In stead, we take a covering of $\hat B_{\alpha}$. 
It is possible to cover $\hat B_{\alpha}$ by 
two contractible metric balls
$\{ B_{\alpha\beta}\}_{\beta=1}^2$.
Thus we have a collection of four contractible metric balls
$\{ B_{\alpha\beta}\}$, which is an essential covering of
$T^2_{\epsilon}$.
Although it is not a usual covering of $T^2_{\epsilon}$, 
deforming and enlarging $B_{\alpha\beta}$ 
by isotopies, 
we obtain a  covering  $\{ \tilde B_{\alpha\beta}\}$ 
of $T^2_{\epsilon}$ by contractible open subsets
$\tilde B_{\alpha\beta}$.
In that sense, the essential covering seems to contain an 
essential property of $T^2_{\epsilon}$.

In the general collapsing case with a lower curvature bound, 
it is believed that
a collapsed space has a certain fiber structure in a generalized
sense such that the fibers shrink to points 
(see \cite{SY:3mfd}, \cite{Ym:4mfd}). 
Although it is not established yet, a fiber which is not
visible yet may shrink to a point with different scales in different
directions, in general. This suggests that
we have to repeat the above process of taking a smaller concentric metric
ball and of covering it by much smaller metric balls
at most $n$-times, $n=\dim M$, to finally reach contractible metric
balls (see Examples \ref{ex:torus} and \ref{ex:nil}). 
In this way, we come to the notion of 
an {\it essential covering} of $M$ with {\it depth} $\le n$.

As illustrated above, an essential covering is not a usual covering of 
$M$, but it contains an essential feature on the complexity of the space $M$.
Actually by deforming and enlarging the balls in the essential
covering by isotopies of $M$ in a 
systematic way, we obtain a real  open covering of $M$.

We define the geometric invariant $\tau_n(M)$ as the minimal number of 
contractible metric balls forming essential coverings of $M$ with
depth $\le n$.
See Section \ref{sec:cover} for a more refined formulation of 
$\tau_n(M)$. In particular, if $M$ is a Riemannian manifold, 
we can replace contractible metric balls by metric balls
homeomorphic to an $n$-disk in the definition of $\tau_n(M)$.

For a positive integer $n$ and $D>0$, we denote by 
$\mathcal A(n,D)$ the isometry  classes of $n$-dimensional compact
Alexandrov spaces with curvature $\ge -1$ and diameter $\le D$.
In this paper, we shall prove

\begin{thm} \label{thm:cover}
For given  $n$ and $D$, there is a positive integer $N(n,D)$
such that  $\tau_n(M)\le N(n,D)$ for all $M$ in $\mathcal A(n,D)$.
\end{thm}

This gives a new geometric  restriction on the spaces in 
$\mathcal A(n,D)$ even in the case of Riemannian manifolds.
A more detailed information on the essential covering of $M$ 
is given in Theorem \ref{thm:refine-cover}.

The minimal number ${\rm cell}(M)$ of open cells needed to cover a 
compact manifold $M$ is an interesting topological invariant.
It is known however that  ${\rm cell}(M)\le\dim M+1$
(see \cite{CLOT:cat}).
This suggests that to have better understanding of the 
complexity of a compact Riemannian manifold or a 
compact Alexandrov space concerning the minimal 
number of some basic subsets needed to cover it,
we have to consider a {\it metric} invariant
rather than a topological invariant.
This is the reason why we mainly consider coverings
by metric balls.

Working with concentric coverings, Theorem \ref{thm:cover}
yields the following 
uniform bound on the total Betti number:

\begin{cor} [\cite{Gr:betti},\cite{LiSh:betti}]\label{cor:sum}
For given  $n$ and $D$, there is a positive integer $C(n,D)$
such that if $M$ is in $\mathcal A(n,D)$, then 
\[
       \sum_{i=0}^n b_i(M;F) \le C(n,D),
\]
where $F$ is any field.
\end{cor}

In the original work \cite{Gr:betti}, Gromov
developed the critical point theory for distance functions to 
obtain an explicit bound on the total Betti numbers for Riemannian manifolds.
The argument in \cite{LiSh:betti} is a natural extension of that
in \cite{Gr:betti} to Alexandrov spaces.
Unfortunately our bound is not explicit. However our approach provides
a conceptually clear view of what the essence of Corollary
\ref{cor:sum}
is like.

For the proof of Theorem \ref{thm:cover}, we use the convergence and
collapsing methods. If a space $M$ in $\mathcal A(n,D)$ does not
collapse,
the stability theorem immediately yields the consequence.
If $M$ collapses to a lower dimensional space, we use 
the rescaling method, which was used in \cite{SY:3mfd} and 
\cite{Ym:4mfd} in some special cases. We first generalize those 
results to the general case. Using this rescaling method, we can 
grasp the proper size of a collapsed ``fiber'' although it is not
visible. 
This enables us to 
have a covering $\{ B_{\alpha_1}\}$ of $M$ such that each ball 
$B_{\alpha_1}$  is,  under some rescaling of metric with the fiber size,
close to 
a complete noncompact Alexandrov
space $Y_1$ of nonnegative curvature with $\dim Y_1>\dim X$
for the pointed Gromov-Hausdorff topology.
If the ``fiber'' uniformly shrinks to a point, the new convergence 
$B_{\alpha_1}\to Y_1$ 
does not collapse. In the other collapsing case of 
$\dim Y_1 <n$, 
we again grasp the size of a new ``fiber''
in the collapsing $B_{\alpha_1}\to Y_1$ 
with the help of the rescaling method. 
From this, we see that a much smaller concentric subball 
$\hat B_{\alpha_1}$ of
$B_{\alpha_1}$, which is isotopic to $B_{\alpha_1}$, can be covered by 
small metric balls $\{ B_{\alpha_1\alpha_2}\}_{\alpha_2}$
whose number is uniformly bounded  such that
each  $B_{\alpha_1\alpha_2}$ is,  under some rescaling with the size
of a new ``fiber'',  
close to a complete noncompact  Alexandrov
space $Y_2$ of nonnegative curvature with $\dim Y_2>\dim Y_1$
for the pointed Gromov-Hausdorff topology.
Repeating this process at most $n-\dim X$ times and using 
the stability theorem,
we finally get an essential covering of $M$ by contractible metric balls,
as required.

Corollary \ref{cor:sum} follows almost directly from Theorem 
\ref{thm:cover} and the topological lemma of \cite{Gr:betti}.
Actually we formulate and prove a more general result for every subset
of an $n$-dimensional complete Alexandrov space 
in terms of $\delta$-content (see Theorem \ref{thm:closed}).

\begin{ack}
   I would like to thank Vitali Kapovitch for 
  bringing the paper \cite{Kp:non-collapse} to my attention.
\end{ack}

% \input{prelim}

% prelim.tex
\section{Preliminaries} \label{sec:prelim}
%\bigskip

\noindent

We refer to Burago, Gromov and Perelman \cite{BGP}  for the basic materials 
on Alexandrov spaces with curvature bounded below.

Let $M$ be an  Alexandrov space with 
curvature bounded below, say $\ge \kappa$.
For two points $x$ and $y$ in $M$, a minimal geodesic joining 
$x$ to $y$ is denoted by $xy$ for simplicity. 
For any geodesic triangle 
$\Delta xyz$ in $M$ with vertices $x, y$ and $z$, we denote by 
$\tilde\Delta xyz$ a {\it comparison triangle} 
in the $\kappa$-plane $M_{\kappa}^2$, 
the simply connected complete surface with constant curvature 
$\kappa$.
The angle between the geodesics $xy$ and $yz$ in $M$ is denoted by 
$\angle xyz$, and the corresponding angle of $\tilde\Delta xyz$
by $\tilde\angle xyz$.
It holds that 
\[
   \angle xyz \ge \tilde\angle xyz.
\]
Let  $\Sigma_p=\Sigma_p(M)$ denote the space of directions 
at $p\in M$.
Let $K_p=K_p(M)$ be the tangent cone at $p$ with vertex $o_p$,
the Euclidean cone over $\Sigma_p$.
For a closed set $A\subset M$ and $p\in M\setminus A$, we denote by 
$A^{\prime}=A^{\prime}_p$ the subset of $\Sigma_p$ consisting 
of all the directions of minimal geodesics from $p$ to $A$.

From now on, we assume that $M$ is finite-dimensional.
It is known that $\Sigma_p$ (resp. $K_p$) is a $(n-1)$-dimensional 
compact (resp. $n$-dimensional complete noncompact) Alexandrov 
space with curvature $\ge 1$ (resp. curvature $\ge 0$), where
$n=\dim M$.

It is well-known that as $r\to 0$, $(\frac{1}{r} M,p)$ converges
to $(K_p,o_p)$ with respect to the  pointed  Gromov-Hausdorff topology,
where  $\frac{1}{r} M$ denotes the rescaling of the original 
distance of $M$ multiplied by  $\frac{1}{r}$.

%% All geodesics are assumed to have unit speed.
%% For a compact set $A\subset X^k$ and $\epsilon>0$, let
%% $\beta_A(\epsilon)$ denote the maximal number of points in $A$
%% having distance $\ge\epsilon$. 

We denote by $\mathcal A_p(n)$ 
the isometry  classes of $n$-dimensional complete 
pointed Alexandrov spaces $(M,p)$ with curvature $\ge -1$.

The following results play crucial roles in this paper.

\begin{thm}[\cite{GLP},\cite{GrKPS}] \label{thm:precpt}
$\mathcal A(n,D)$ (resp. $\mathcal A_p(n)$) is relatively compact 
with respect to the Gromov-Hausdorff distance
(resp. the pointed  Gromov-Hausdorff topology).
\end{thm}

Consider the distance function $d_p(x)=d(p,x)$ from a 
point $p\in M$. A point $q\neq p$ is a {\it critical point} of $d_p$
if $\tilde\angle pqx \le \pi/2$ for all $x\in M$.

For $0< r < R$, $A(p;r,R)$ denotes the closed annulus 
$\bar B(p,R) - B(p,r)$, where $B(p,r)$ is the open metric ball around
$p$  of radius $r$.

\begin{lem}[\cite{GS:sphere},\cite{Gr:betti},\cite{Pr:alex2}] \label{lem:iso}
If $d_p$ has no critical points on $A(p;r,R)$, then 
$A(p;r,R)$ is homeomorphic to $\partial B(p,r)\times [0,1]$.
\end{lem}

\begin{thm}[\cite{Pr:alex2}, cf. \cite{Kp:stability}] \label{thm:stab}
Let an infinite sequence $(M_i,p_i)$ in $\mathcal A_p(n)$ 
converge to a space $(M, p)$ in $\mathcal A_p(n)$  
with respect to the pointed  
Gromov-Hausdorff topology.  Take an $r>0$ such that 
there are no critical points of $d_p$ on 
$B(p, r)-\{ p\}$.  Then $B(p_i, r)$ is homeomorphic to both $B(p,r)$
and $K_p$ for large $i$.
\end{thm}

%  \input{rescale}

%\rescale.tex 
\section{Rescaling metrics} \label{sec:rescale}

Let a sequence $(M_i,p_i)$ in $\mathcal A_p(n)$ converge to a 
pointed Alexandrov space $(X,p)$ with curvature $\ge -1$
with respect to the pointed Gromov-Hausdorff topology.
It is a fundamental problem to find topological relation 
between $B(p_i,r)$ and $B(p,r)$ for a small but fixed 
positive number $r$ and large $i$. 

In the case when
$\dim X=n$, take $r>0$ so that the distance
function $d_p$ has no critical points on $B(p,r)-\{ p\}$.
Then Theorem \ref{thm:stab} shows that $B(p_i,r)$ is homeomorphic to 
$B(p,r)$ for large $i$.

In this section, from now on, we consider the collapsing  case when 
$1\le \dim X \le n-1$.
Since we are concerned with the topology of a neighborhood 
of $p_i$, we may assume 

\begin{asmp} \label{asmp:not-disk}
$B(\tilde p_i,r)$ is not homeomorphic to an $n$-disk for any 
$\tilde p_i$ with $d(p_i,\tilde p_i)\to 0$  and for any sufficiently 
large $i$.
\end{asmp}

The following is a generalization of the Key lemma 3.6
in \cite{SY:3mfd} and Theorem 4.1 in \cite{Ym:4mfd}.

\begin{thm} \label{thm:rescal}
Under Assumption \ref{asmp:not-disk}, 
there exist $\hat p_i\in B(p_i,r)$ and a sequences $\delta_i\to 0$ 
such that 
\begin{enumerate}
  \item $d(\hat p_i, p_i)\to 0;$
  \item  $d_{\hat p_i}$ has no critical points on 
   $A(\hat p_i;R\delta_i, r)$ for every $R\ge 1$ and large $i$ 
   compared to $R$. In particular, 
   $B(\hat p_i,r)$ is homeomorphic to $B(\hat p_i,R\delta_i);$ 
  \item  for any limit $(Y,y_0)$ of $(\frac{1}{\delta_i}M_i, \hat p_i)$, 
     we have  $\dim Y\ge \dim X + 1$.
 \end{enumerate}
\end{thm}

The essential idea of the proof of Theorem \ref{thm:rescal} is the
same  as in \cite{SY:3mfd}. In \cite{SY:3mfd} however,
we had to suppose that the function $\hat f:K_p\to \R$ constructed there takes a strict 
local maximum at the vertex $o_p$ of $K_p$. 
Since this does not hold in general, we must modify the construction.
Some simplification of the proof  is also made here.
\par\medskip
For positive numbers $\theta$ and $\epsilon$ with 
$\epsilon\ll \theta\le\pi/100$,
take a positive number $r=r(p,\theta,\epsilon)$ such that 
\begin{enumerate}
\item $\angle xpy -\tilde\angle xpy < \epsilon$
      for every $x,y\in \partial B(p,2r);$
\item $\{ y_p'\}_{y\in\partial B(p,2r)}$ is $\epsilon$-dense in
  $\Sigma_p$.
\end{enumerate}
Note that the above $(2)$ implies that there are no critical points of
$d_p$ on $B(p,r)-\{ p\}$.
Let $\{ x_{\alpha}\}_{\alpha}$ be a $\theta r$-discrete maximal system in 
$\partial B(p,r)$.
For a small positive number $\epsilon$, take an $\epsilon r$-discrete 
maximal system $\{ x_{\alpha\beta}\}_{\beta}$, $1\le \beta\le N_{\alpha}$,
in $B(x_{\alpha},\theta r)\cap \partial B(p,r)$.
Let $\xi_{\alpha\beta}\in\Sigma_{p}$ be the direction of geodesic
$px_{\alpha\beta}$. Note that 
$\{ \xi_{\alpha\beta}\}_{\beta}$ is $\epsilon/2$-discrete. 
A standard covering argument implies that 
\begin{equation}
   N_{\alpha} \ge {\rm const}\left(\frac{\theta}{\epsilon}\right)^{\dim X -1}.
                   \label{eq:net}
\end{equation}

We consider the following functions $f_{\alpha}$ and $f$ on $M$:
\[
   f_{\alpha}(x) = \frac{1}{N_{\alpha}} \sum_{\beta=1}^{N_{\alpha}} 
                                d(x_{\alpha\beta}, x), \qquad
  f(x)= \min_{\alpha} f_{\alpha}(x).
\]
A similar construction was made in \cite{Kp:non-collapse}
to define a strictly concave function  on a neighborhood of a given 
point of an Alexandrov space.
The effectiveness of the use of those functions was suggested to the 
author by Vitali Kapovitch.

\begin{lem} \label{lem:str-max}
For every $x\in B(p,r/2)$, we have $f(x) \le r - d(p,x)/2$.
In particular, the restriction of $f$ to $B(p,r/2)$ has a strict maximum at $p$.
\end{lem}

\begin{proof}
Take $y\in\partial B(p,r)$ and $x_{\alpha}$ with $\angle xpy <\epsilon$
and $d(y,x_{\alpha})<\theta r$.
It follows that $\angle xpx_{\alpha\beta} < 5\theta$.
Let $\gamma:[0,d]\to X$ be a minimal geodesic joining $p$ to $x$.
By the curvature assumption with trigonometry, we see that
$\angle x\gamma(t)x_{\alpha\beta} < \pi/4$. The first variation formula then 
implies that $d(x_{\alpha\beta}, x) \le r - d(p,x)/2$, and therefore
$f(x)\le f_{\alpha}(x) \le r - d(p,x)/2$.
\end{proof}

\begin{proof}[Proof of Theorem \ref{thm:rescal}]
Take  a $\mu_i$-approximation 
\[
    \phi_i : B(p,1/\mu_i)\to B(p_i,1/\mu_i),
\]
with $\phi_i(p)=p_i$, where $\mu_i\to 0$ as $i \to \infty$.  
Let $x_{\alpha\beta}^i:=\phi_i(x_{\alpha\beta})$, and 
define the functions $f_{\alpha}^i$ and $f^i$ on $M_i$ 
by
\[
    f_{\alpha}^i(x) = \frac{1}{N_{\alpha}} \sum_{\beta=1}^{N_{\alpha}} 
                                d(x_{\alpha\beta}^i, x), \qquad
  f^i(x)= \min_{\alpha} f_{\alpha}^i(x).
\]
Note that $f_{\alpha}^i\circ \phi_i \to f_{\alpha}$ and
$f^i\circ \phi_i \to f$.
By Lemma \ref{lem:str-max}, there is a point 
$\hat p_i\in B(p_i,r/2)$ such that 
\begin{enumerate}
 \item $(p_i,\hat p_i)\to 0;$
 \item the restriction of $f^i$ to $B(p_i,r/3)$ takes a maximum at 
       $\hat p_i$.
\end{enumerate}
Consider the distance function $d_{\hat p_i}$. 
By Assumption \ref{asmp:not-disk}, there is a critical point of 
$d_{\hat p_i}$ in $B(\hat p_i,r)$.  Let 
$\delta_i$
be the maximum distance between $\hat p_i$ and
the critical point set of $d_{\hat p_i}$ within  $B(\hat p_i,r)$.
Note that $\delta_i\to 0$.
Let $\hat q_i$ be a critical point of $d_{\hat p_i}$ within
$B(\hat p_i,r)$ realizing $\delta_i$.
We may assume that 
$(\frac{1}{\delta_i} M_i, \hat p_i)$ converges to a complete
noncompact pointed Alexandrov space $(Y,y_0)$ with nonnegative 
curvature. Let $z_0\in Y$ be the 
limit of $\hat q_i$ under this convergence.
We denote by $\hat d =\frac{1}{\delta_i}d$ the distance of 
$\frac{1}{\delta_i} M_i$. Consider the function 
\[
  h_{\alpha\beta}^i(x):=
   \hat d(x_{\alpha\beta}^i,x)-\hat d(x_{\alpha\beta}^i,\hat p_i),
\]
which is $1$-Lipschitz, and bounded on every bounded set.
Therefore passing to a subsequence, we may assume that
$h_{\alpha\beta}$ converges to a $1$-Lipschitz function 
$h_{\alpha\beta}$ on $Y$.

%% \begin{slem}
%% $h_{\alpha\beta}$ is concave on $Y$.
%% \end{slem}

%% \begin{proof}
%% Take a sequence $a_m\to\infty$ of posiitve numbers, and 
%% consider the metric balls 
%% $B_{\alpha\beta}^i(a_m):=B(x_{\alpha\beta}^i,R_{\alpha\beta}^i)$, where
%% $R_{\alpha\beta}^i:= \hat d(\hat p_i,x_{\alpha\beta}^i)-a_m$.
%% Passing to subsequences together with diagonal argument, we may assume 
%% that $B_{\alpha\beta}^i(a_m)$ converges to closed subsets
%% $B_{\alpha\beta}(a_m)$ of $Y$ under the convence
%% $(\frac{1}{\delta_i} M_i, \hat p_i) \to (Y,y_0)$.
%% Note that 
%% $h_{\alpha\beta}(x)=d(x, B_{\alpha\beta}(a_m) - d(y_0,
%% B_{\alpha\beta}(a_m))$ for every $x$ and large $m$.
%% Since $d(y_0,B_{\alpha\beta})\to\infty$, 
%% $h_{\alpha\beta}$ is a generalized Busemann function.
%% Thus the conclusion follows from the nonnegativity of the curvature of $Y$.
%% \end{proof}

Let 

\begin{align*}
   h_{\alpha} = \frac{1}{N_{\alpha}} \sum_{\beta=1}^{N_{\alpha}} 
                               h_{\alpha\beta}, & \qquad
   h = \min_{\alpha} h_{\alpha}, \\
   h_{\alpha}^i = \frac{1}{N_{\alpha}} \sum_{\beta=1}^{N_{\alpha}} 
                               h_{\alpha\beta}^i, & \qquad
   h^i = \min_{\alpha} h_{\alpha}^i.
\end{align*}
Since $h^i=[f^i - f^i(\hat p_i)]/\delta_i$, 
$h$ takes a maximum at $y_0$.
Let $x_{\alpha\beta}(\infty)$  denote the the element of the ideal 
boundary $Y(\infty)$ of $Y$ 
defined by the limit ray, say $\gamma_{\alpha\beta}$, from $y_0$ 
of the geodesic $\hat p_i x_{\alpha\beta}^i$
under the convergence 
$(\frac{1}{\delta_i} M_i, \hat p_i)\to (Y,y_0)$. 
Let $v_{\alpha\beta}\in\Sigma_{y_0}$  and $v\in\Sigma_{y_0}$ denote the direction 
of $\gamma_{\alpha\beta}$ and $yz$ respectively.
Since $z_0$ is a critical point of $d_{y_0}$,
we have 
$\tilde\angle y_0z_0x_{\alpha\beta}(\infty)\le\pi/2$.
Since $Y$ has nonnegative curvature, it follows that
$\angle(v, v_{\alpha\beta}) \ge 
       \tilde\angle z_0 y_0 x_{\alpha\beta}(\infty)\ge\pi/2$,
for every $\alpha$ and $\beta$.
Choosing $\alpha$ with 
$h_{y_0}'(v)=(h_{\alpha})_{y_0}'(v)$, 
we obtain
\[
0  \ge h_{y_0}'(v)=\frac{1}{N_{\alpha}} \sum_{\beta=1}^{N_{\alpha}} 
                   - \cos \angle(v, v_{\alpha\beta}),
\]
and therefore $\angle(v, v_{\alpha\beta}) = \pi/2$.
Since 
\begin{align*}
 \angle(v_{\alpha\beta}, v_{\alpha\beta'}) 
     & = \lim_{t\to 0}\tilde\angle\gamma_{\alpha\beta}(t) y_0
                                      \gamma_{\alpha\beta'}(t) \\
     & \ge\tilde\angle x_{\alpha\beta}p x_{\alpha\beta'} \\
     & \ge \epsilon/2,
\end{align*}
for every $1\le \beta\neq\beta'\le N_{\alpha}$,
$\{ v_{\alpha\beta}\}_{\beta=1}^{N_{\alpha}}$ is 
$\epsilon/2$-discrete in $\partial B(v,\pi/2)\subset\Sigma_{y_0}$.
Since $\Sigma_{y_0}$ has curvature $\ge 1$, 
there is an expanding map from $\partial B(v,\pi/2)$ to the unit  sphere
$S^{\dim Y-2}(1)$.  
It follows that
\begin{equation}
    N_{\alpha} \le \text{\rm const}\,\epsilon^{-(\dim Y - 2)}.
        \label{eq:net2}
\end{equation}
Since this holds for any sufficiently small $\epsilon$, 
%with $N_{\alpha}\sim\text{\rm const}(\frac{\theta}{\epsilon})^{\dim X - 1}$,
from \eqref{eq:net} and  \eqref{eq:net2} we can conclude $\dim Y\ge \dim X + 1$.
This completes the proof of Theorem \ref{thm:rescal}.
\end{proof}

%  \input{system}

%system.tex 
\section{Isotopy covering systems and essential coverings}\label{sec:cover}

Let $M$ be a compact $n$-dimensional Alexandrov 
space with curvature bounded below.
For an open metric ball $B$ of $M$, we denote by $\lambda B$ the concentric 
ball of radius $\lambda r$.  We call a concentric ball $\hat B\subset B$ 
an {\it isotopic subball} of $B$ if there is a homeomorphism 
$M\to M$ sending $\hat B$ onto $B$ and leaving the outside of 
a neighborhood of $\bar B$ fixed.
For instance, this is the case when $d_p$ has no critical points 
on $\bar B - \hat B$ (Lemma\ref{lem:iso}).

Consider the following system 
$\mathcal B =\{ B_{\alpha_1\cdots\alpha_k}\}$
consisting of open metric balls $B_{\alpha_1\cdots\alpha_k}$  of $M$,
where the indices $\alpha_1, \ldots, \alpha_k$ range over
\begin{align*}
  1\le\alpha_1\le N_1, & \quad  1\le\alpha_2\le N_2(\alpha_1),\\
   1\le\alpha_k\le & N_k(\alpha_1\cdots\alpha_{k-1}),
\end{align*}
and 
$1\le k\le \ell$ for some $\ell$ depending on the choice of
the indices $\alpha_1,\alpha_2,\ldots$. 
Note that the range of $\alpha_k$ also depends on 
$\alpha_1\cdots\alpha_{k-1}$.
Let $A$ be the set  of all multi-indices
$\alpha=\alpha_1\cdots\alpha_{k}$
such that $B_{\alpha_1\cdots\alpha_{k}}\in\mathcal B$.
For each $\alpha=\alpha_1\cdots\alpha_{k}\in A$, put
$|\alpha|:= k$ and call it the length of $\alpha$.

Let $X$ be a subset of $M$. 
We call $\mathcal B$ an {\it isotopy covering system} 
of  $X$ if it satisfies the following:
\begin{enumerate}
 \item  $\{ B_{\alpha_1}\}_{\alpha_1=1}^{N_1}$ covers  $X;$
 \item $B_{\alpha_1\cdots\alpha_{k-1}} \supset 
                   B_{\alpha_1\cdots\alpha_{k}};$
 \item $\{ B_{\alpha_1\cdots\alpha_k}\}_{\alpha_k=1}^
                {N_k(\alpha_1\cdots\alpha_{k-1})}$ is a covering of 
   an isotopic subball  $\hat B_{\alpha_1\cdots\alpha_{k-1}}$ of 
  $B_{\alpha_1\cdots\alpha_{k-1}};$
 \item there is a uniform bound $d$ such that $|\alpha|\le d$ 
  for all $\alpha\in A$.
\end{enumerate}
We call $N_1$ the {\it first degree} of the system $\mathcal B$, and 
$N_k(\alpha_1\cdots\alpha_{k-1})$ the {\it $k$-th degree} of $\mathcal B$
with respect to $\alpha_1\cdots\alpha_{k-1}$.

Let $\hat A$ be the set  of all maximal multi-indices
$\alpha_1\cdots\alpha_{\ell}$  in $A$ in the 
sense that there are no  $\alpha_{\ell+1}$ with 
$B_{\alpha_1\cdots\alpha_{\ell}\alpha_{\ell+1}}\in\mathcal B$.
Then $\mathcal U:=\{ B_{\alpha}\}_{\alpha\in \hat A}$ is called 
an {\it essential covering} of $B$. 
In other words, $\mathcal U$ is the collection of the metric balls
lying on the bottom of the system $\mathcal B$.

We show that the essential covering 
$\mathcal U=\{ B_{\alpha}\}_{\alpha\in \hat A}$
produces a covering 
$\tilde {\mathcal U}=\{ \tilde B_{\alpha}\}_{\alpha\in \hat A}$ of $X$
such that $\tilde B_{\alpha}$ is homeomorphic (actually isotopic)
to  $B_{\alpha}$.
Let 
$h_{\alpha_1\cdots\alpha_{k-1}}: M \to M$ be a homeomorphism 
sending  $\hat B_{\alpha_1\cdots\alpha_{k-1}}$ onto
$B_{\alpha_1\cdots\alpha_{k-1}}$
and leaving the outside of 
a neighborhood of $\bar B_{\alpha_1\cdots\alpha_{k-1}}$ fixed.
For each $\alpha=\alpha_1\cdots\alpha_{\ell}\in \hat A$, consider 
the open set
\[
\tilde B_{\alpha} := 
   h_{\alpha_1}\circ h_{\alpha_1\alpha_2}\circ\cdots\circ
         h_{\alpha_1\cdots\alpha_{\ell-1}}(B_{\alpha}).
\]
For each $1\le\alpha_1\le N_1$, let $A(\alpha_1)$ be the set of 
all multi-indices $\alpha\in A$  of the forms 
$\alpha=\alpha_1\cdots\alpha_k$ whose leading term is equal to 
$\alpha_1$ and $k\ge 2$.
From construction, we have 
\[
   B_{\alpha_1}\subset \bigcup_{\alpha\in A(\alpha_1)}
                  \tilde B_{\alpha} 
\]
and therefore
$\tilde{\mathcal U}= \{ \tilde B_{\alpha}\}_{\alpha\in \hat A}$ provides a 
covering of $X$. 

We call 
\[
       d_0 := \max_{\alpha\in \hat A} |\alpha|
\]
the {\it depth} of both  $\mathcal B$ and 
$\mathcal U$. 
Note that if $d_0=1$, then  $\mathcal B=\mathcal U$ is a usual covering of $X$.

Let ${\mathcal C}(n)$ be the set of all isometry classes of 
the Euclidean cone over $(n-1)$-dimensional compact Alexandrov spaces 
with curvature $\ge 1$.
We say that $\mathcal B$ and $\mathcal U$
are {\it modeled} on ${\mathcal C}(n)$ if each $B_{\alpha}$ in
$\mathcal U$
is homeomorphic to a space in ${\mathcal C}(n)$.

For any positive integer $d$, 
we denote by $\tau_d(X)$ the minimal number of metric balls 
forming an essential covering $\mathcal U$  of $X$ with 
depth $\le d$ modeled on 
$\mathcal C(n)$.
Note that $\tau_{d_1}(B)\ge \tau_{d_2}(B)$ if $d_1\le d_2$.

 For open metric ball $B$ of $M$ having a proper isotopic subball, 
 we set
 \[
      \tau_d^*(B) = \min_{\hat B} \tau_{d}(\hat B),
 \]
 where $\hat B$ runs over all isotopic subballs of $B$.
If $B$ itself is homeomorphic to a space in $\mathcal C(n)$,
we define
\[
     \tau_0(B)=\tau_0^*(B)=1.
\]

From definition, we immediately have

\begin{lem} \label{lem:tau}
Suppose that $X$ is covered by metric balls
$\{ B_{\alpha_1}\}_{\alpha_1=1}^{N_1}$ having proper isotopic subballs.
Then we have 
\[
    \tau_{d+1}(X) \le \sum_{\alpha_1=1}^{N_1} \tau_d^*(B_{\alpha_1}).
\]
\end{lem}

\begin{ex} \label{ex:torus}
For a positive number $\epsilon$,
let us consider the flat torus
\[
    T^n_{\epsilon}=S^1(1)\times S^1(\epsilon)\times S^1(\epsilon^2)
                    \times\cdots\times S^1(\epsilon^{n-1}).
\]
An obvious observation similar to that in the introduction shows 
$\tau_n(T^n)\le 2^n$. Note that  
$\lim_{\epsilon\to 0}\tau_d(T^n_{\epsilon})=\infty$ 
for every $1\le d\le n-1$.
\end{ex}

\begin{ex} \label{ex:nil}
Let $N$ be an $n$-dimensional  simply connected Lie group, and 
$\frak n$ its Lie algebra. Take a triangular basis 
$x_1,\ldots,x_n$ of $\frak n$ in the sense that
$[x_i, x]\in \frak n_{i-1}$ for every $x\in\ell$, where
$\frak n_{i-1}$ is spanned by $x_1,\ldots,x_{i-1}$. 
For $\epsilon>0$, put  
$\epsilon_i:=\epsilon^{n^{n-i}}$, and define the inner product on
$\frak n$ by
\[
    ||x||_{\epsilon}^2 = 
           \epsilon_1^2 a_1^2 + \cdots + \epsilon_n^2 a_n^2,
\]
for $x=\sum a_ix_i$. We equip $N$ the corresponding left invariant 
metric $g_{\epsilon}$.
For a given uniform discrete subgroup $\Gamma$ of $N$,
consider the quotient $M_{\epsilon}:=(\Gamma\backslash N, g_{\epsilon})$.
Note that the sectional curvature of $M_{\epsilon}$ is uniformly
bounded and 
$\delta_{\epsilon}^1:=\diam(M_{\epsilon}) \to 0$ as $\epsilon\to 0$
(see \cite{Gr:almost}). 
Now under the rescaling of metric  
$\frac{1}{\delta_{\epsilon}^1} M_{\epsilon}$
collapses to a circle. We then have a fibration
\[
      \Gamma_1\backslash N_1 \to M_{\epsilon} \to S^1,
\]
with a nilmanifold $\Gamma_1\backslash N_1$ as fiber.
Thus $M_{\epsilon}$ can be covered by two thin metric balls 
$B_{\alpha_1}$, $\alpha_1\in \{ 1, 2\}$,
each of which is homeomorphic to $\Gamma_1\backslash N_1\times [0,1]$.
Let $\delta_{\epsilon}^2:=\diam(\Gamma_1\backslash N_1)$.
Under the rescaling of metric  $\frac{1}{\delta_{\epsilon}^1} B_{\alpha_1}$
collapses to $S^1\times\R$. 
Now an isotopic subball $\hat B_{\alpha_1}$ 
of $B_{\alpha_1}$ has a fibration
\[
      \Gamma_2\backslash N_2\to \hat B_{\alpha_1} \to S^1\times [0,1],
\]
with a nilmanifold $\Gamma_2\backslash N_2$ as fiber.
Thus $\hat B_{\alpha_1}$ can be 
covered by two metric balls 
$B_{\alpha_1\alpha_2}$, $\alpha_2\in \{ 1, 2\}$,
each of which is homeomorphic to $\Gamma_2\backslash N_2\times [0,1]^2$.
Repeating this, we finally have $\tau_n(M_{\epsilon})\le 2^n$. Note that  
$\lim_{\epsilon\to 0}\tau_d(M_{\epsilon})=\infty$ 
for every $1\le d\le n-1$.
\end{ex}
 
Let $A(n)$ denote the set of all isometry classes of $n$-dimensional 
complete Alexandrov spaces with curvature $\ge -1$.
Theorem \ref{thm:cover} is an immediate consequence of 
the following 

\begin{thm} \label{thm:refine-cover}
For given $n$ and $D>0$, there are constants $C_n$ and $C_n(D)$ such 
that for every metric ball $B$ of radius $\le D$ in $M\in\mathcal A(n)$, 
there is an isotopy covering
system  $\mathcal B=\{ B_{\alpha_1\cdots\alpha_k}\}$
of $B$ with depth $\le n$  modeled on 
$\mathcal C(n)$ such that 
\begin{enumerate}
 \item the first degree  $\le C_n(D);$
 \item any other higher degree $\le C_n$.
\end{enumerate}
In  particular $\tau_n(B)\le C_n(D)(C_n)^{n-1}$.
\end{thm}
\bigskip
We first prove the local version of Theorem \ref{thm:refine-cover}.

\begin{lem} \label{lem:pt-cover}
There is a positive number $C_n$ satisfying the 
following: For  a given  infinite sequence $(M_i,p_i)$ in $\mathcal A_p(n)$
with  $\inf \diam(M_i) > 0$,
there is a subsequence $(M_j,p_j)$ for which 
we have  a positive number $r>0$ and $\hat p_j\in M_j$ with 
$d(p_j,\hat p_j)\to 0$ such that $\tau_{n-1}^*(B(\hat p_j, r))\le C_n$.
\end{lem}

\begin{proof}
We prove it by contradiction. If the conclusion does not hold, we 
would have an infinite sequence $(M_i,p_i)$ in $\mathcal A_p(n)$ such 
that for every $r>0$ and every $\hat p_i\in M_i$ with
$d(p_i,\hat p_i)\to 0$, we have $\tau_{n-1}^*(B(\hat p_j, r))\to \infty$
for any subsequence $\{ j\}$ of $\{ i\}$.
By Theorem \ref{thm:precpt}, we have a subsequence  $\{ j\}$ such that 
$(M_j,p_j)$ converges to a pointed space $(X,p)$.
Set $k=\dim X$.

We claim that  $\tau_{n-k}^*(B(\hat p_j, r))\le C$ for some $r>0$ and 
constant $C$ independent of $j$, where $\hat p_j$ is a point of $M_j$ with
$d(p_j,\hat p_j)\to 0$.
Since this is a contradiction, this will complete the proof.

We prove the claim  by the reverse induction on $k$.
If $k=n$, then Theorem \ref{thm:stab}
shows that there is an $r>0$ such that $B(p_j, r)$ is homeomorphic 
to  $K_p$, yielding $\tau_0^*(B(p_j, r))=1$.
Therefore together with the diameter assumption, we only have to 
investigate the case $1\le \dim X\le n-1$. 
Suppose the claim holds for $\dim X=k+1, \ldots, n$, and 
consider the case of $\dim X=k$.
Take $r=r_p>0$, $\hat p_j$ and $\delta_j\to 0$
as in Theorem \ref{thm:rescal}. 
Namely passing to a subsequence, we may assume that 
$(\frac{1}{\delta_j}M_j, \hat p_j)$
converges to a pointed complete noncompact nonnegatively curved 
space $(Y,y_0)$ with  $\dim Y\ge \dim X + 1$
such that $B(\hat p_j,R\delta_j)$ is an isotopic subball of 
$B(\hat p_j,r)$  for every $R\ge 1$ and large $j$ compared to $R$.
%
%% If $k=n-1$, we apply the stability theorem to 
%% the convergence $(\frac{1}{\delta_j}M_j, \hat p_j)\to (Y,y_0)$
%% to see the following:
%% For each $z\in B(y_0, 2)$, there are  $z^j\in
%% (\frac{1}{\delta_j}M_j, \hat p_j)$  and $r_z>0$ such that 
%% $B(z, r_z)=B(z^j, r_z;\frac{1}{\delta_j}M_j)$ is homeomorphic to 
%% $B(z, r_z)$ for large enough $j$. Hence
%% $\tau_n^*(B(z^j, r_z))=1$.
%% By compactness, there are finitely many points
%% $z_{\alpha}\in B(y_0, 2)$ and $z_{\alpha}^j\in M_j$
%% converging to $z_{\alpha}$ together with $r_{\alpha}>0$ such that
%% $\cup B(z_{\alpha}, r_{\alpha}/2)\supset B(y_0, 2)$ and 
%% $\tau_n^*(B(z_{\alpha}^j, r_{\alpha}))=1$.
%% Note that 
%% $\cup B(z_{\alpha}^j, r_{\alpha})\supset 
%% B(\hat p_j, 2;\frac{1}{\delta_j}M_j)$ for large $j$.
%% Thus  we can conclude  $\tau_n^*(B(\hat p_j, r))\le C$.
%
Applying the induction hypothesis to the convergence
$(\frac{1}{\delta_j}M_j, \hat p_j) \to (Y,y_0)$, we have the
following:
For each $z\in B(y_0, 2)$, there are  $z^j\in
(\frac{1}{\delta_j}M_j, \hat p_j)$  and $r_z>0$ such that 
$\tau_{n-k-1}^*(B(z^j, r_z;\frac{1}{\delta_j}M_j))\le C$ for some constant
$C$ independent of $j$.  
By compactness, there are finitely many points
$z_{\alpha}\in B(y_0, 2)$ and $z_{\alpha}^j\in M_j$
converging to $z_{\alpha}$ together with $r_{\alpha}>0$ such that
\[
    \bigcup B(z_{\alpha}, r_{\alpha}/2)\supset B(y_0, 2), \qquad
 \tau_{n-k-1}^*(B(z_{\alpha}^j, r_{\alpha};\frac{1}{\delta_j}M_j))
                \le C_{\alpha}.
\]
Note that 
$\cup B(z_{\alpha}^j, r_{\alpha};\frac{1}{\delta_j}M_j)\supset 
B(\hat p_j, 2;\frac{1}{\delta_j}M_j)$ for large $i$.
Thus  we can conclude  
\begin{align*}
   \tau_{n-k}^*(B(\hat p_j, r)) 
      & \le \tau_{n-k}(B(\hat p_j, 2\delta_j)) \\
      & \le \sum_{\alpha} \tau_{n-k-1}^*(B(\hat z_{\alpha}^j, r_{\alpha})) \\
      & \le  \sum C_{\alpha} <\infty.
\end{align*}
\end{proof}

\begin{proof}[Proof of Theorem \ref{thm:refine-cover}]
The proof is by contradiction.
If the conclusion does not hold, we would have an infinite
sequence of metric balls $B_i$ of spaces $M_i \in \mathcal A(n)$ such that 
for every essential covering system 
$\mathcal B^i$ of $B_i$ with depth $\le n$ modeled on 
$\mathcal C(n)$, either 
$\lim\inf N_1^i=\infty$ or $\lim\inf N_k^i > C_n$,
where $N_1^i$, $N_k^i$ are the degrees of $\mathcal B^i$,  and 
$C_n$ is the positive constant given in Lemma
\ref{lem:pt-cover}.
Let $p_i$ be the center of $B_i$.
By  Theorem \ref{thm:precpt}, we may assume that 
$(M_i, p_i)$ converges to a pointed complete Alexandrov  space 
$(X, p)$ with curvature $\ge -1$
with respect to the pointed Gromov-Hausdorff topology.
We may also assume that $B_i$ converges to a metric ball $B$
around $p$ under this convergence.
If $X$ is a point, we rescale the metric of $M_i$ so that the new 
diameter is equal to $1$. Thus we may assume that 
$1\le\dim X\le n$. Applying Lemma  \ref{lem:pt-cover} to
the convergence $B_i\to B$,  we obtain 
finitely many points $\{ x_{\alpha} \}_{\alpha=1}^N$ of $B$
and positive numbers $r_{\alpha}$ with 
$B \subset \cup B(x_{\alpha}, r_{\alpha}/2)$ such that 
for a subsequence $\{ j\}$ of $\{ i\}$, 
we get $p_{\alpha}^j\in M_j$ converging to $x_{\alpha}$
with $\tau_{n-1}^*(B(p_{\alpha}^j,r_{\alpha}))\le C_n$
for every $1\le\alpha\le N$.
Together with the covering $\{ B(p_{\alpha}^j,r_{\alpha})\}_{\alpha}$
of $M_j$, this enables us to
obtain an essential covering system $\mathcal B^j$ of $B_j$
with depth $\le n$  modeled on 
$\mathcal C(n)$ such that 
$N_1^j \le N$ and  $N_k^j \le C_n$.
This is a contradiction.
\end{proof}

%% An $n$-dimensional complete noncompact Alexandrov space  $Y$  
%% with nonnegative curvature is called {\it smoothable} if 
%% it is the limit space of a sequence of $n$-dimensional 
%% complete Riemannian manifolds with a lower sectional curvature 
%% bound, for the pointed Gromov-Hausodorff topology
%% (see \cite{Kp:non-collapse}).
\begin{rem} \label{rem:mfd}
Let $\mathcal M(n)$ denote the subfamily of $\mathcal A(n)$
consisting of Riemannian manifolds.
By  Theorem \ref{thm:refine-cover}, each metric ball of radius $\le D$ in
$M\in\mathcal M(n)$ has an essential covering 
with depth $\le n$ modeled on 
$\mathcal C(n)$ whose number is uniformly bounded.
In this case, one can easily check from the proof that
each metric ball in the essential covering
is homeomorphic to an $n$-disk. Namely, for 
$\mathcal M(n)$, we can take the single 
$n$-dimensional Euclidean space $\R^n$ as the model family 
in stead of $\mathcal C(n)$.
\end{rem}

\begin{rem} \label{rem:delta}
Let $\delta >0$ be given. Under the situation of Theorem
\ref{thm:refine-cover}, if we restrict ourselves to metric balls of
radii $<\delta$, 
we can construct an isotopy covering system 
$\mathcal B=\{ B_{\alpha_1\cdots\alpha_k}\}$
of $B$ with depth $\le n$  modeled on 
$\mathcal C(n)$ such that 
\begin{enumerate}
 \item  the radius of $B_{\alpha_1}$  is less than $\delta$ for 
        every $1\le \alpha_1\le N_1;$
 \item the first degree  $N_1\le C_n(D,\delta)$ for some uniform 
       constant $C_n(D,\delta);$
 \item any other higher degree $\le C_n$.
\end{enumerate}
In  particular $\tau_n(B)\le C_n(D,\delta)(C_n)^{n-1}$.
\end{rem}

Parhaps Examples \ref{ex:torus} and \ref{ex:nil} will be 
ones of maximal case.

 \begin{conj}\label{conj:max}
 Let $M$ be an $n$-dimensional compact 
 Alexandrov space with nonnegative curvature.
 Then $\tau_n(M)\le 2^n$.
 \end{conj}

%  \input{sum}

%\sum.tex 
\section{Betti numbers} \label{sec:sum}

In this section, we apply Theorem \ref{thm:refine-cover} to 
prove Corollary \ref{cor:sum}.
We consider homology groups with any coefficient field $F$.
Let $\beta(\quad)$ denote the total Betti number for simplicity.

We make use of the following machinery in \cite{Gr:betti}, whose proof
is  based on Leray's spectral sequence.

\begin{lem}[Topological lemma (\cite{Gr:betti})] \label{lem:top}
Let $B_{\alpha}^i$, $1\le \alpha \le N$, $0\le i\le n+1$,
be open subsets of an $n$-dimensional space $M$, with 
\[
  B_{\alpha}^0\subset B_{\alpha}^1\subset \cdots B_{\alpha}^{n+1},
\]
and set $A^i := \cup_{i=1}^N B_{\alpha}^i$.
Let $I_{+}$ denote the set of all multi-indices $(\alpha_1, \ldots,
\alpha_m)$
with $1\le \alpha_1 < \ldots <\alpha_m \le N$ and with non-empty
intersection
$\cap_{j=1}^m B_{\alpha_j}^{n+1}$.
For each $\mu=(\alpha_1, \ldots, \alpha_m)\in I_+$, let 
$f_{\mu}^i:H_{*}(B_{\alpha_1}^i \cap \cdots \cap B_{\alpha_m}^i) \to
H_{*}(B_{\alpha_1}^{i+1} \cap \cdots \cap B_{\alpha_m}^{i+1})$ be the 
inclusion homomorphism. Then the rank of the inclusion homomorphism
$H_*(A^0)\to H_*(A^{n+1})$ is bounded above by the sum
\begin{equation*}
   \sum_{0\le i\le n, \mu\in I_+}
              \text{\rm rank} \,f_{\mu}^i
%  \sum
%      \begin{Sb}
%            0\le i\le n \\  \mu\in I_+
 %%      \end{Sb}
%%              \text{\rm rank} \,f_{\mu}^i
\end{equation*}
\end{lem}
\bigskip

For any subset $X\subset M$ and $\delta\ge 0$, let 
$U_{\delta}(X) := \{ x\,|\,d(x,X)\le \delta \}$.
We define 
{\it $\delta$-content}, denoted by $\delta$-${\rm cont}(X)$ of $X$ 
as the rank of the 
inclusion homomorphism
$H_*(X)\to H_*(U_{\delta}(X))$. 
Observe that $0$-${\rm cont}(X)=\beta(X)$ may be infinite. 
However we have

\begin{thm} \label{thm:closed}
For given  $n$, $D >0$ and $\delta>0$, there is a positive integer 
$C(n,D,\delta)$
such that if $X$ is a subset of diameter $\le D$ in an $n$-dimensional 
complete Alexandrov space $M$ with curvature $\ge -1$,
then 
\[
      \text{$\delta$-${\rm cont}$}(X) \le  C(n,D,\delta).
\]
\end{thm}

Corollary \ref{cor:sum} is a direct consequence of Theorem \ref{thm:closed}.
Although it is not explicitly stated in \cite{Gr:betti} or 
\cite{Abr:lowII}, Theorem \ref{thm:closed} also follows from the 
methods there. Below we give the proof of Theorem \ref{thm:closed}
based on Theorem \ref{thm:refine-cover}.

%\begin{proof}[Proof of Corollary \ref{cor:sum}]
For a  subset $X$ of diameter less than $D$ in a space $M \in \mathcal
A(n)$,  let $B$ an open metric $D$-ball in $M$
containing $X$.  
For $\delta>0$, take 
an isotopy covering system
$\mathcal B =\{ B_{\alpha_1\cdots\alpha_k}\}$ of $B$ with depth $\le
n$  modeled on $\mathcal C(n)$
satisfying the conclusion of Theorem \ref{thm:refine-cover} and
Remark \ref{rem:delta} such that the radii of $B_{\alpha_1}$
are less than $10^{-(n+2)}\delta$ for all $1\le \alpha_1\le N_1$.
To apply Lemma \ref{lem:top}, we let 
$\lambda_i:= 10^i$ for $0\le i\le n+1$, and put
\[
  B_{\alpha_1\cdots\alpha_k}^i := \lambda_i
  B_{\alpha_1\cdots\alpha_k}.
\]
In view of  the conclusion $(2)$ of Theorem \ref{thm:rescal}, 
we may assume that 
\begin{enumerate}
\item $B_{\alpha_1\cdots\alpha_k}^{n+1}\subset 
          B_{\alpha_1\cdots\alpha_{k-1}};$
 \item $B_{\alpha_1\cdots\alpha_k}^{i}$ is an isotopic subball of 
       $B_{\alpha_1\cdots\alpha_{k}}^{i+1}$,
\end{enumerate}
for each $1\le\alpha_k\le N_k(\alpha_1\cdots\alpha_{k-1})$ and 
$0\le i\le n+1$.

Let $\mathcal U=\{ B_{\alpha}\}_{\alpha\in\hat A}$ be the essential
covering of $B$ associated with $\mathcal B$.

\begin{lem} \label{lem:ball-sys}
For every  $\alpha=\alpha_1\cdots\alpha_{\ell}\in\hat A$ and 
every $1\le k\le \ell$ we have
\[
    \beta( B_{\alpha_1\cdots\alpha_k})\le C_n.
\]
\end{lem}

\begin{proof}
We prove it by the reverse induction on $k$.
The case $k=\ell$ is clear since $B_{\alpha_1\cdots\alpha_{\ell}}$ is 
contractible.
Suppose the conclusion $\beta( B_{\alpha_1\cdots\alpha_{k+1}})\le C_n$
for all $1\le \alpha_{k+1}\le N_{k+1}$.
Let $\hat B_{\alpha_1\cdots\alpha_{k}}$ be the isotopic subball of 
$B_{\alpha_1\cdots\alpha_{k}}$ such that 
\[
  \hat B_{\alpha_1\cdots\alpha_{k}}\subset
   \bigcup_{\alpha_{k+1}=1}^{N_{k+1}} B_{\alpha_1\cdots\alpha_{k+1}}.
    %% \subset 2 \hat B_{\alpha_1\cdots\alpha_{k}}.
\]
Since 
$(\alpha_1,\ldots,\alpha_{k})$ is fixed, we put
\begin{align*}
  \hat B := \hat B_{\alpha_1\cdots\alpha_{k}},
                   \qquad & B := B_{\alpha_1\cdots\alpha_{k}}, \\
  B_{\alpha}:= B_{\alpha_1\cdots\alpha_{k}\alpha}, 
        \qquad & B_{\alpha}^i := \lambda_i B_{\alpha},
\end{align*}
for each $1\le\alpha\le N_{k+1}$.
%% From the conclusion $(2)$ of Theorem \ref{thm:rescal}, 
%% we may assume that 
%% \[
%%   B_{\alpha}^{n+1}\subset B, \qquad
%%  \bar B_{\alpha}^{i+1}- B_{\alpha}^i\simeq \partial B_{\alpha}^i\times
%%  [0,1].
%% \]
Let $A^i=\bigcup_{\alpha=1}^{N_{k+1}} B_{\alpha}^i$.
From the inclusions $\hat B\subset A^0\subset A^{n+1}\subset B$,
we have $\beta(\hat B)=\beta(B)\le\text{\rm rank of}\,[H_*(A^0)\to
H_*(A^{n+1})]$.
Let $I_+$ denote the set of multi-indices of intersection
for the covering $\{ B_{\alpha}^{n+1}\}_{\alpha}$.
For each $\mu=(\gamma_1,\ldots,\gamma_m)\in I_+$, let 
$B_{\gamma_s}$ have minimal radius among $\{ B_{\gamma_j}\}_{j=1}^m$.
Let $f_{\mu}^i:H_{*}( B_{\gamma_1}^i \cap \cdots \cap B_{\gamma_m}^i)) \to
H_{*}( B_{\gamma_1}^{i+1} \cap \cdots \cap B_{\gamma_m}^{i+1})$ be the 
inclusion homomorphism.
From the inclusions
\[
     B_{\gamma_1}^i \cap \cdots \cap B_{\gamma_m}^i  \subset
     B_{\gamma_s}^i \subset \frac{1}{2} B_{\gamma_s}^{i+1}
      \subset  B_{\gamma_1}^{i+1} \cap \cdots \cap B_{\gamma_m}^{i+1},
\]
we have 
\begin{align*}
  \text{\rm rank}(f_{\mu}^i) 
     & \le \text{\rm rank of}\,[H_*(B_{\gamma_s}^i)\to H_*(\frac{1}{2}
               B_{\gamma_s}^{i+1})]\\
     & =  \beta(B_{\gamma_s}) \\
     & \le C_n.
\end{align*}
Lemma \ref{lem:top} then shows $\beta(B)\le (n+1)2^{C_n} C_n$.
\end{proof}

\begin{proof}[Proof of Theorem \ref{thm:closed}]
Without loss of generality, we may assume that 
$\{ B_{\alpha_1}\}_{\alpha_1=1}^{N(X)}$  is a covering of 
$X$ for some $N(X)$ with $N(X)\le N_1$.
By Lemma \ref{lem:ball-sys}, we have in particular
$\beta(B_{\alpha_1})=\beta(B_{\alpha_1}^i)\le C_n$ 
for all $1\le \alpha_1\le N(X)$ and $0\le i\le n+1$. 
%% Actually we take concentric coverings
%% $\{ B_{\alpha_1}^i \}_{\alpha_1=1}^{N_1}$ of $M$, $0\le i\le n+1$,
%% where $B_{\alpha_1}^i=\lambda_i B_{\alpha_1}$ as before.
%%  Lemma \ref{lem:ball-sys} essentially says that 
%% $\beta(B_{\alpha_1}^i)\le C_n$.
Therefore applying Lemma \ref{lem:top} to 
the concentric coverings
$\{ B_{\alpha_1}^i \}_{\alpha_1=1}^{N(X)}$ of $X$ together with 
\[
   X \subset \bigcup_{\alpha_1=1}^{N(X)} B_{\alpha_1}
       \subset \bigcup_{\alpha_1=1}^{N(X)} B_{\alpha_1}^{n+1}
        \subset U_{\delta}(X),
\]
we have
\[
   \text{$\delta$-{\rm cont}}(X)\le (n+1)2^{C_n(D,\delta)} C_n.
\]
This completes the proof of Theorem  \ref{thm:closed}.
\end{proof}

For a subset $X$ of a metric space, 
we define the {\it homological injectivity radius} of $X$, denoted
by  ${\rm hom.inj}(X)$,  as the supremum of $\delta\ge 0$
such that the inclusion homomorphism $H_*(X)\to H_*(U_{\delta}(X))$ is 
injective for any coefficient field. 

The following is an immediate consequence of Theorem \ref{thm:closed}.

\begin{cor}
For a space $M$ in $\mathcal A(n)$,  let $X_i$ be a sequence of 
subsets of $M$ with $\lim\beta(X_i)=\infty$.
Then one of the following must occur:
\begin{enumerate}
 \item $\lim\inf\, {\rm hom.inj}(X_i)=0;$
 \item $\lim\sup\, \diam(X_i)=\infty$.
\end{enumerate}
\end{cor}

%\bibliographystyle{amsplain}
% \bibliography{all}

\end{document}